\documentclass{birkjour}
\usepackage{amsmath, amsthm, amscd, amsfonts, amssymb, graphicx, color}
\usepackage{amsrefs}
\usepackage[bookmarksnumbered, colorlinks, plainpages]{hyperref}
\usepackage[mathscr]{eucal}
\hypersetup{colorlinks=true, linkcolor=red, anchorcolor=green, citecolor=blue, urlcolor=red, filecolor=magenta, pdftoolbar=true}

\newtheorem{theorem}{Theorem}[section]
\newtheorem{lemma}[theorem]{Lemma}
\newtheorem{proposition}[theorem]{Proposition}
\newtheorem{corollary}[theorem]{Corollary}
\theoremstyle{definition}
\newtheorem{definition}[theorem]{Definition}

\theoremstyle{remark}
\newtheorem{remark}[theorem]{Remark}
\numberwithin{equation}{section}

\begin{document}
\setcounter{page}{1}

\title[Orlicz spaces associated to a quasi-Banach function space]{Orlicz spaces associated to a quasi-Banach function space. Applications to vector measures and interpolation}

\author{R. del Campo}
\address{%
Dpto. Matem\'atica Aplicada I\\
E.T.S.I.A. Carretera de Utrera, km 1\\
41013 Sevilla\\
Spain}
\email{rcampo@us.es}

\address{%
\\
A. Fern\'{a}ndez, F. Mayoral and F. Naranjo\\
Dpto. Matem\'atica Aplicada II\\
E.T.S.I. Camino de los Descubrimientos s/n\\
41092 Sevilla\\
Spain}

\author{A. Fern\'andez}
\email{afcarrion@us.es}

\author{F. Mayoral}
\email{mayoral@us.es}

\author{F. Naranjo}
\email{naranjo@us.es}



\begin{abstract}
The Orlicz spaces $X^{\Phi}$ associated to a quasi-Banach function space $X$ are defined by replacing the role of the space $L^1$ by $X$ in the classical construction of Orlicz spaces. Given a vector measure $m,$ we can apply this construction to the spaces $L^1_w(m),$ $L^1(m)$ and $L^1(\|m\|)$ of integrable functions (in the weak, strong and Choquet sense, respectively) in order to obtain the known Orlicz spaces $L^{\Phi}_w(m)$ and $L^{\Phi}(m)$ and the new ones $L^{\Phi}(\|m\|).$ Therefore, we are providing a framework where dealing with different kind of Orlicz spaces in a unified way. Some applications to complex interpolation are also given.
\end{abstract}

\maketitle

\section{Introduction}
\label{intro}

The Banach lattice $L^1(m)$ of integrable functions with respect to a vector measure $m$ (defined on a $\sigma$-algebra of sets and with values in a Banach space) has been systematically studied during the last 30 years and it has proved to be a efficient tool to describe the optimal domain of operators between Banach function spaces (see \cite{ORS} and the references therein). The Orlicz spaces $L^{\Phi}(m)$ and $L^{\Phi}_w(m)$ associated to $m$ were introduced in \cite{Del} and they have recently shown in \cite{CFMN2} their utility in order to characterize compactness in $L^1(m).$

On the other hand, the quasi-Banach lattice $L^1(\|m\|)$ of integrable functions (in the Choquet sense) with respect to the semivariation of $m$ was introduced in \cite{FMN}. Some properties of this space and their corresponding $L^p(\|m\|)$ with $p>1$ have been obtained, but in order to achieve compactness results in $L^1(\|m\|)$ we would need to dispose of certain Orlicz spaces related to $L^1(\|m\|).$

In \cite{JPU} some generalized Orlicz spaces $X_{\Phi}$ have been obtained by repla\-cing the role of the space $L^1$ by a Banach function space $X$ in the classical construction of Orlicz spaces. Moreover, the spaces $X$ they consider are always supposed to possess the $\sigma$-Fatou property. However, these Orlicz spaces do not cover our situation since:
\begin{itemize}
\item[$\bullet$] the space $L^1(\|m\|)$ is only a quasi-Banach function space, and 
\item[$\bullet$] in most of the time $L^1(m)$ lacks the $\sigma$-Fatou property.
\end{itemize}
Thus, the purpose of this work is to provide a construction of certain Orlicz spaces $X^{\Phi}$ valid for the case of $X$ being an arbitrary quasi-Banach function space (in general without the $\sigma$-Fatou property), with the underlying idea that it can be applied simultaneously to the spaces $L^1(\|m\|)$ and $L^1(m)$ among others. In a subsequent paper \cite{preprint} we shall employ these Orlicz spaces $L^{1}(\|m\|)^{\Phi}$ and their main properties here derived in order to study compactness in $L^1(\|m\|).$

The organization of the paper goes as follows: Section 2 contains the preliminaries which we will need later. Section 3 contains a discussion of completeness in the quasi-normed context without any additional hypothesis on $\sigma$-Fatou property. Section 4 is devoted to introduce the Orlicz spaces $X^{\Phi}$ associated to a quasi-Banach function space $X$ and obtain their main properties. In Section 5, we show that the construction of the previous section allows to capture the Orlicz spaces associated to a vector measure and we take advantage of its generality to introduce the Orlicz spaces associated to its semivariation. Finally, in Section 6 we present some applications of this theory to compute their complex interpolation spaces.

\section{Preliminaries}

Throughout this paper, we shall always assume that $\Omega$ is a nonempty set, $\Sigma$ is a $\sigma$-algebra of subsets of $\Omega,$ $\mu$ is a finite positive measure defined on $\Sigma$ and $L^0(\mu)$ is the space of ($\mu$-a.e. equivalence classes of) measurable functions $f:\Omega\rightarrow\mathbb{R}$ equipped with the topology of convergence in measure.

Recall that a {\em quasi-normed space\/} is any real vector space $X$ equipped with a {\em quasi-norm\/}, that is, a function $\|\cdot\|_{X} : X\rightarrow [0,\infty)$ which satisfies the following axioms:
\begin{itemize}
\item[(Q1)] $\|x\|_{X}=0$ if and only if $x=0.$
\item[(Q2)] $\|\alpha x\|_{X} = |\alpha| \|x\|_{X},$ for $\alpha\in\mathbb{R}$ and $x\in X.$
\item[(Q3)] There exists $K\geq 1$ such that $\left\|x_1+x_2\right\|_{X} \leq K\left(\left\|x_1\right\|_{X} + \left\|x_2\right\|_{X}\right),$ for all $x_1,x_2\in X.$
\end{itemize}
The constant $K$ in (Q3) is called {\em a quasi-triangle constant\/} of $X.$ Of course if we can take $K=1,$ then $\|\cdot\|_{X}$ is a norm and $X$ is a normed space. A {\em quasi-normed function space\/} over $\mu$ is any quasi-normed space $X$ satisfying the following properties:
\begin{itemize}
\item[(a)] $X$ is an ideal in $L^0(\mu)$ and a quasi-normed lattice with respect to the $\mu$-a.e. order, that is, if $f\in L^0(\mu),$ $g\in X$ and $|f|\leq |g|$ $\mu$-a.e., then $f\in X$ and $\|f\|_{X} \leq \|g\|_{X}.$
\item[(b)] The characteristic function of $\Omega,$ $\chi_{\Omega},$ belongs to $X.$
\end{itemize}
If, in addition, the quasi-norm $\|\cdot\|_{X}$ happens to be a norm, then $X$ is called a {\em normed function space\/}. Note that, with this definition, any quasi-normed function space over $\mu$ is continuously embedded into $L^0(\mu),$ as it is proved in \cite[Proposition~2.2]{ORS}.

\begin{remark}
Many of the results that we will present in this paper are true if we assume that the measure space $\left(\Omega,\Sigma,\mu\right)$ is $\sigma$-finite. In this case, the previous condition (b) must be replaced by
\begin{itemize}
\item[(b')] The characteristic function $\chi_{A}$ belongs to $X$ for all $A\in \Sigma$ such that $\mu(A)<\infty.$
\end{itemize}
Nevertheless we prefer to present the results in the finite case for clarity and simplicity in the proofs.
\end{remark}

We say that a quasi-normed function space $X$ has the {\em $\sigma$-Fatou property\/} if for any positive increasing sequence $(f_n)_n$ in $X$ with  $\displaystyle \sup_{n}\|f_n\|_{X}<\infty$ and converging pointwise $\mu$-a.e. to a function $f,$ then $f\in X$ and $\displaystyle \|f\|_{X} = \sup_{n}\|f_n\|_{X}.$ And a quasi-normed function space $X$ is said to be {\em $\sigma$-order continuous\/} if for any positive increasing sequence $(f_n)_n$ in $X$ converging pointwise $\mu$-a.e. to a function $f\in X,$ then $\left\|f-f_n\right\|_{X}\to 0.$

A complete quasi-normed function space is called a {\em quasi-Banach function space} (briefly q-B.f.s.). If, in addition, the quasi-norm happens to be a norm, then $X$ is called a {\em Banach function space\/} (briefly B.f.s.). It is known that if a quasi-normed function space has the $\sigma$-Fatou property, then it is complete and hence a q-B.f.s. (see \cite[Proposition 2.35]{ORS}) and that inclusions between q-B.f.s. are automatically continuous (see \cite[Lemma 2.7]{ORS}).

Given a countably additive vector measure $m:\Sigma\rightarrow Y$ with values in a real Banach space $Y,$ there are several ways of constructing $q$-B.f.s. of integrable functions. Let us recall them briefly. The {\em semivariation\/} of $m$ is the finite subadditive set function defined on $\Sigma$ by
$$
\|m\|(A):=\sup\left\{\left|\langle m, y^*\rangle\right|(A) : y^*\in B_{Y^*} \right\},
$$
where $\left|\langle m, y^*\rangle\right|$ denotes the variation of the scalar measure $\langle m, y^*\rangle:\Sigma\rightarrow\mathbb{R}$ given by $\langle m, y^*\rangle(A):= \langle m(A), y^*\rangle$ for all $A\in\Sigma,$ and $B_{Y^*}$ is the unit ball of $Y^*,$ the dual of $Y.$ A set $A\in\Sigma$ is called {\em $m$-null\/} if $\|m\|(A)=0.$ A measure $\mu:= \left|\langle m, y^*\rangle\right|,$ where  $y^*\in B_{Y^*},$ that is equivalent to $m$ (in the sense that $\|m\|(A)\to 0$ if and only if $\mu(A)\to 0$) is called a {\em Rybakov control measure\/} for $m.$ Such a measure always exists (see \cite[Theorem 2, p.268]{DU}). Let $L^0(m)$ be the space of ($m$-a.e. equivalence classes of) measurable functions $f:\Omega\longrightarrow\mathbb{R}.$ Thus, $L^0(m)$ and $L^0(\mu)$ are just the same whenever $\mu$ is a Rybakov control measure for $m.$

A measurable function $f:\Omega\longrightarrow\mathbb{R}$ is called {\em weakly integrable\/} (with respect to $m$) if $f$ is integrable with respect to $|\langle m, y^*\rangle|$ for all $y^*\in Y^*.$ A weakly integrable function $f$ is said to be {\em integrable\/} (with respect to $m$) if, for each $A\in\Sigma$ there exists an element (necessarily unique) $\displaystyle\int_A f \, dm \in Y,$ satisfying
$$
\left\langle \int_A f \, dm, y^* \right\rangle = \int_A f \, d\langle m, y^*\rangle, \quad y^*\in Y^*.
$$
Given a measurable function  $f:\Omega\longrightarrow\mathbb{R},$ we shall also consider its distribution function (with respect to the semivariation of the vector measure $m$)
$$
\|m\|_{f}: t\in [0,\infty)\rightarrow \|m\|_f(t):=\|m\|\left( \left[|f|>t\right] \right)\in [0,\infty),
$$
where $\left[|f|>t\right] := \left\{ w\in\Omega : |f(w)|>t \right\}.$ This distribution function is bounded, non-increasing and right-continuous.

Let $L^1_w(m)$ be the space of all ($m$-a.e. equivalence classes of) weakly integrable functions, $L^1(m)$ the space of all ($m$-a.e equivalence classes of) integrable functions and $L^1(\|m\|)$ the space of all ($m$-a.e. equivalence classes of) measurable functions $f$ such that its distribution function $\|m\|_{f}$ is Lebesgue integrable in $(0,\infty).$ Letting $\mu$ be any Rybakov control measure for $m,$ we have that $L^1_w(m)$ becomes a B.f.s. over $\mu$ with the $\sigma$-Fatou property when endowed with the norm
$$
\|f\|_{L^1_w(m)} := \sup\left\{  \int_{\Omega} |f| \, d|\langle m, y^*\rangle|  : y^*\in B_{Y^*} \right\}.
$$
Moreover, $L^1(m)$ is a closed $\sigma$-order continuous ideal of $L^1_w(m).$ In fact, it is the closure of $\mathscr{S}(\Sigma),$ the space of simple functions supported on $\Sigma.$ Thus, $L^1(m)$ is a $\sigma$-order continuous B.f.s. over $\mu$ endowed with same norm (see \cite[Theorem 3.7]{ORS} and \cite[p.138]{ORS})). It is worth noting that space $L^1(m)$ does not generally have the $\sigma$-Fatou property. In fact, if $L^1(m) \neq L^1_w(m),$ then $L^1(m)$ does not have the $\sigma$-Fatou property. See \cite{CR} for details.

On the other hand, $L^1(\|m\|)$ equipped with the quasi-norm
$$
\|f\|_{L^1(\|m\|)}:= \int_{0}^{\infty} \|m\|_{f}(t)\, dt.
$$
is a q-B.f.s. over $\mu$ with the $\sigma$-Fatou property (see \cite[Proposition 3.1]{CFMN1}) and it is also $\sigma$-order continuous (see \cite[Proposition 3.6]{CFMN1}). We will denote by $L^{\infty}(m)$ the B.f.s. of all ($m$-a.e. equivalence classes of) essentially bounded functions equipped with the essential sup-norm.

\section{Completeness of quasi-normed lattices}

In this section we present several characterizations of completeness which will be needed later. We begin by recalling one of them valid for general quasi-normed spaces (see \cite[Theorem 1.1]{Mal}).

\begin{theorem} 
\label{Maligranda completeness}
Let $X$ be a quasi-normed space with a quasi-triangle constant $K.$ The following conditions are equivalent:
\begin{itemize}
\item [(i)] $X$ is complete.
\item [(ii)] For every sequence $(x_n)_n\subseteq X$ such that $\displaystyle\sum_{n=1}^{\infty}K^n\|x_n\|_{X}<\infty$ we have $\displaystyle\sum_{n=1}^{\infty} x_n \in X.$
In this case, the inequality $\displaystyle\left\|\sum_{n=1}^{\infty} x_n \right\|_{X} \leq K \sum_{n=1}^{\infty} K^n \|x_n\|_{X}$ holds.
\end{itemize}
\end{theorem}

The next result is a version of Amemiya's Theorem (\cite[Theorem 2, p.290]{KaAk}) for quasi-normed lattices.

\begin{theorem} 
\label{Amemiya completeness}
Let $X$ be a quasi-normed lattice. The following conditions are equivalent:
\begin{itemize}
\item [(i)] $X$ is complete.
\item [(ii)] For any positive increasing Cauchy sequence $(x_n)_n$ in $X$ there exists $\displaystyle\sup_{n} x_n \in X.$
\end{itemize}
\end{theorem}

\begin{proof}
(i) $\Rightarrow$ (ii) is evident because the limit of increasing convergent sequences in a quasi-normed lattice is always its supremum.

\noindent
(ii) $\Rightarrow$ (i) Let $(x_n)_n$ be a positive increasing Cauchy sequence in $X.$ It is sufficient to prove that $(x_n)_n$ is convergent in $X$ for $X$ being complete (see, for example \cite[Theorem 16.1]{AlBu}). By hypothesis, there exists $x:=\displaystyle\sup_{n} x_n \in X.$ We have to prove that $(x_n)_n$ converges to $x$ and for this it is enough the convergence of a subsequence of $(x_n)_n.$ So, let us take a subsequence of $(x_n)_n,$ that we still denote by $(x_n)_n,$ such that $\|x_{n+1}-x_n\|_{X}\leq \displaystyle\frac{1}{K^n n^3},$ for all $n\in\mathbb{N}$ where $K$ is a quasi-triangle constant of $X.$ Thus, the sequence $y_n:=\displaystyle\sum_{i=1}^{n} i (x_{i+1}-x_i)$ is positive, increasing and Cauchy. Indeed, given $m>n,$ we have
$$
\left\|y_m-y_n\right\|_{X} \leq \sum_{i=n+1}^{m} i K^{i-n} \|x_{i+1}-x_{i}\|_{X} \leq \frac{1}{K^n} \sum_{i=n+1}^{m} \frac{1}{i^2}\leq \sum_{i=n+1}^{m} \frac{1}{i^2}.
$$
Applying (ii) again, we deduce that there exists $y:=\displaystyle\sup_{n} y_n\in X.$ Moreover, given $n\in\mathbb{N},$ we have
\begin{eqnarray*}
n(x-x_n) & = & n \left(\sup_{m>n} x_{m+1} - x_n\right) \ = \ n \sup_{m>n} (x_{m+1}-x_n)
\\
& = &  n \sup_{m>n} \sum_{i=n}^{m} (x_{i+1}-x_i) \ \leq \ \sup_{m>n} y_n \ = \ y.
\end{eqnarray*}
Therefore, $0\leq x-x_n \leq \displaystyle\frac{1}{n}y$ \ and hence $\|x-x_n\|_{X}\leq \displaystyle\frac{1}{n}\|y\|_{X}\to 0.$
\end{proof}

Applying Theorem \ref{Amemiya completeness} to the sequence of partial sums of a given sequence, we see that completeness in quasi-normed lattices can still be cha\-racterized by a Riesz-Fischer type property.

\begin{corollary}
Let $X$ be a quasi-normed lattice with a quasi-triangle constant $K.$ The following conditions are equivalent:
\begin{itemize}
\item [(i)] $X$ is complete.
\item [(ii)] For every positive sequence $(x_n)_n\subseteq X$ such that $\displaystyle\sum_{n=1}^{\infty}K^n\|x_n\|_{X}<\infty$ there exists $\displaystyle\sup_{n} \sum_{i=1}^{n} x_i \in X.$
\end{itemize}
\end{corollary}

\section{Orlicz spaces $X^{\Phi}$}

In this section we introduce the Orlicz spaces $X^{\Phi}$ associated to a quasi-Banach function space $X$ and a Young function $\Phi$ and obtain their main properties.

Recall that a {\em Young function\/} is any function $\Phi:[0,\infty)\rightarrow[0,\infty)$ which is strictly increasing, continuous, convex, $\Phi(0)=0$ and $\displaystyle\lim_{t\to\infty}\Phi(t)=\infty.$ A Young function $\Phi$ satisfies the following useful inequalities (which we shall use without explicit mention) for all $t\geq 0$:
$$
\left\{
\begin{array}{ccl}
\Phi(\alpha t) \leq \alpha \, \Phi(t) & \mbox{ if } & 0\leq\alpha\leq 1,
\\
\Phi(\alpha t) \geq \alpha \, \Phi(t) & \mbox{ if } & \alpha\geq 1.
\end{array}
\right.
$$
In particular, from the second of the previous inequalities it follows that for all  $t_0>0$ there exists $C>0$ such that $\Phi(t)\geq C t$ for all $t\geq t_0.$ For a given $t_0>0,$ just take $\displaystyle C:=\frac{\Phi(t_0)}{t_0}>0$ and observe that 
$\displaystyle
\Phi(t) = \Phi\left(t_0\frac{t}{t_0}\right)\geq \frac{t}{t_0}\Phi(t_0) = C t$
for all $t\geq t_0.$

Moreover, it is easy to prove using the convexity of $\Phi$ that
\begin{equation}\label{quasi-subaditivity}
\Phi\left( \sum_{n=1}^{N} t_n \right) \leq \sum_{n=1}^{N} \frac{1}{2^n \alpha^n} \Phi(2^n \alpha^n t_n)
\end{equation}
for all $N\in\mathbb{N},$ $\alpha\geq 1$ and $t_1,\dots,t_N\geq 0.$

A Young function $\Phi$ has the {\em $\Delta_2$-property\/}, written $\Phi\in\Delta_2,$ if there exists a constant $C>1$ such that $\Phi(2t)\leq C\Phi(t)$ for all $t\geq 0.$ Equivalently, $\Phi\in\Delta_2$ if for any $c>1$ there exists $C>1$ such that $\Phi(ct)\leq C\Phi(t),$ for all $t\geq 0.$

\begin{definition}
Let $\Phi$ be a Young function. Given a quasi-normed function space $X$ over $\mu,$ the corresponding (generalized) {\em Orlicz class\/} $\widetilde{X}^{\Phi}$ is defined as the following set of ($\mu$-a.e. equivalence classes of) measurable functions:
$$
\widetilde{X}^{\Phi} := \left\{ f\in L^0(\mu) : \Phi(|f|)\in X \right\}.
$$
\end{definition}

\begin{proposition} 
\label{OC properties}
Let $\Phi$ be a Young function and $X$ be a quasi-normed function space over $\mu.$ Then, $\widetilde{X}^{\Phi}$ is a solid convex set in $L^0(\mu).$ Moreover, $\widetilde{X}^{\Phi} \subseteq X.$
\end{proposition}

\begin{proof}
Let $f,g\in \widetilde{X}^{\Phi}$ and $0\leq\alpha\leq 1.$ According to the convexity and monotonicity properties of $\Phi$ we have
$$
\Phi(|\alpha f + (1-\alpha)g|) \leq \alpha \Phi(|f|) + (1-\alpha)\Phi(|g|) \in X.
$$
The ideal property of $X$ yields $\Phi(|\alpha f + (1-\alpha)g|)\in X$ which means that 
$$
\alpha f + (1-\alpha)g \in \widetilde{X}^{\Phi}
$$ 
and proves the convexity of $\widetilde{X}^{\Phi}.$ Clearly, $\widetilde{X}^{\Phi}$ is solid, since $|h|\leq |f|$ implies that $\Phi(|h|)\leq \Phi(|f|)\in X,$ for any $h\in L^0(\mu).$ Moreover, since $\Phi$ is convex function, there exists $C>0$ such that $\Phi(t)\geq C t,$ for all $t > 1.$ Thus, for all $f\in \widetilde{X}^{\Phi},$
$$
|f| = |f|\chi_{\left[|f|>1\right]} + |f|\chi_{\left[|f|\leq 1\right]} \leq \frac{1}{C}\Phi\left(|f|\chi_{\left[|f| > 1\right]}\right) + \chi_{\Omega} \leq  \frac{1}{C}\Phi\left(|f|\right) + \chi_{\Omega} \in X,
$$
which gives $f\in X.$
\end{proof}

\begin{definition}
Let $\Phi$ be a Young function. Given a quasi-normed function space $X$ over $\mu,$ the corresponding (generalized) {\em Orlicz space\/} $X^{\Phi}$ is defined as the following set of ($\mu$-a.e. equivalence classes of) measurable functions:
$$
X^{\Phi}:= \left\{ f\in L^0(\mu) : \exists \, c >0 : \frac{|f|}{c}\in\widetilde{X}^{\Phi} \right\}.
$$
\end{definition}

\begin{proposition} 
\label{OS properties}
Let $\Phi$ be a Young function and $X$ be a quasi-normed function space over $\mu.$ Then, $X^{\Phi}$ is a linear space, an ideal in $L^0(\mu)$ and $\widetilde{X}^{\Phi} \subseteq X^{\Phi}\subseteq X.$
\end{proposition}

\begin{proof}
Let $f,g\in X^{\Phi}$ and $\alpha\in\mathbb{R}.$ Then, there exist $c_1, c_2>0$ such that $\displaystyle\frac{|f|}{c_1}, \displaystyle\frac{|g|}{c_2}\in \widetilde{X}^{\Phi}.$ Setting $c:=\max\{c_1,c_2\}$ and using the convexity of $\widetilde{X}^{\Phi}$ we have
$$
\frac{|f+g|}{2c} \leq \frac{1}{2}\frac{|f|}{c} + \frac{1}{2}\frac{|g|}{c} \leq \frac{1}{2}\frac{|f|}{c_1} + \frac{1}{2}\frac{|g|}{c_2} \in \widetilde{X}^{\Phi}
$$
and hence $\displaystyle\frac{|f+g|}{2c} \in \widetilde{X}^{\Phi}$ since $\widetilde{X}^{\Phi}$ is solid, which proves that $f+g\in X^{\Phi}.$ Note that this also implies that $nf\in X^{\Phi}$ for any $n\in\mathbb{N}.$ Taking $n_0\in\mathbb{N}$ such that $|\alpha|\leq n_0,$ it follows that there exists $c_0>0$ such that $\displaystyle\frac{|\alpha f|}{c_0}\leq \frac{n_0|f|}{c_0} \in \widetilde{X}^{\Phi},$ which yields $\displaystyle\frac{|\alpha f|}{c_0}\in \widetilde{X}^{\Phi}$ and so $\alpha f\in X^{\Phi}.$

It is evident that $\widetilde{X}^{\Phi} \subseteq X^{\Phi}$ and $X^{\Phi}$ inherits the ideal property from $\widetilde{X}^{\Phi},$ since $|h|\leq |f|$ implies that $\displaystyle\frac{|h|}{c_1} \leq \frac{|f|}{c_1}\in \widetilde{X}^{\Phi}$ for any $h\in L^0(\mu).$ Moreover, taking into account Proposition \ref{OC properties}, we have $\displaystyle\frac{|f|}{c_1}\in\widetilde{X}^{\Phi} \subseteq X$ and so $f\in X$ which proves that $X^{\Phi} \subseteq X.$
\end{proof}

\begin{definition}
Let $\Phi$ be a Young function and $X$ be a quasi-normed function space over $\mu.$ Given $f\in X^{\Phi},$ we define
$$
\|f\|_{X^{\Phi}} := \inf \left\{ k>0 : \frac{|f|}{k}\in\widetilde{X}^{\Phi} \mbox{ with } \left\| \Phi\left(\frac{|f|}{k}\right) \right\|_{X}\leq 1 \right\}.
$$
The functional \mbox{$\|\cdot\|_{X^{\Phi}}$} in $X^{\Phi}$ is called the Luxemburg quasi-norm.
\end{definition}

\begin{proposition} 
\label{LQN properties}
Let $\Phi$ be a Young function and $X$ be a quasi-normed function space (respectively, normed function space) over $\mu.$ Then, \mbox{$\|\cdot\|_{X^{\Phi}}$} is a quasi-norm (respectively, norm) in $X^{\Phi}.$ Moreover, $X^{\Phi}$ equipped with the Luxemburg quasi-norm, is a quasi-normed (respectively, normed) function space over $\mu.$
\end{proposition}

\begin{proof}
First, note that $\|\cdot\|_{X^{\Phi}}:X^{\Phi}\to [0,\infty).$ Given $f\in X^{\Phi},$ there exists $c>0$ such that $\Phi \displaystyle\left(\frac{|f|}{c}\right)\in X.$ Let $M:=\displaystyle\left\|\Phi\left(\frac{|f|}{c}\right)\right\|_X<\infty.$ On the one hand, if $M\leq 1$ then $\|f\|_{X^{\Phi}}\leq c < \infty.$ On the other hand, if $M>1$ then $\displaystyle\Phi\left(\frac{|f|}{Mc}\right)\leq \frac{1}{M}\Phi\left(\frac{|f|}{c}\right)\in X$ and so $\displaystyle\left\|\Phi\left(\frac{|f|}{Mc}\right)\right\|_X \leq \frac{1}{M} \left\|\Phi\left(\frac{|f|}{c}\right)\right\|=1,$ which implies that $\|f\|_{X^{\Phi}}\leq Mc < \infty.$

\noindent
If $f=0,$ then $\displaystyle\left\|\Phi\left(\frac{|f|}{c}\right)\right\|_X=0\leq 1$ for all $c>0$ and so $\|f\|_{X^{\Phi}}=0.$ Now, suppose that $\|f\|_{X^{\Phi}}=0$ and that $\mu\left( \left[f\neq 0\right]\right)>0,$ that is, $\displaystyle\left\|\Phi\left(\frac{|f|}{c}\right)\right\|_X\leq 1$ for all $c>0$ and there exist $\varepsilon>0$ and $A\in\Sigma$ such that $\mu(A)>0$ and $|f|\chi_A \geq \varepsilon\chi_A.$ Given $c>0,$ we have
$\displaystyle
\Phi\left(\frac{\varepsilon}{c}\right)\chi_A \leq \Phi\left(\frac{|f|\chi_A}{c}\right) \leq \Phi\left(\frac{|f|}{c}\right).$
Therefore,
$\displaystyle
\left\|\Phi\left(\frac{|f|}{c}\right)\right\|_X \geq \left\|\Phi\left(\frac{\varepsilon}{c}\right)\chi_A \right\|_X = \Phi\left(\frac{\varepsilon}{c}\right) \|\chi_A\|_X
$
and keeping in mind that $\displaystyle\lim_{t\to\infty}\Phi(t)=\infty,$ we can take $c>0$ such that 
$\displaystyle
\Phi\left(\frac{\varepsilon}{c}\right) \|\chi_A\|_X > 1
$ 
which yields a contradiction.

On the other hand, given $f\in X^{\Phi}$ and $\lambda\in\mathbb{R},$ it is clear that
\begin{eqnarray*}
\|\lambda f\|_{X^{\Phi}} & = &  \inf \left\{ k>0 :  \left\| \Phi\left(\frac{|\lambda f|}{k}\right) \right\|_{X}\leq 1 \right\} \\
& = & \inf \left\{ k>0 :  \left\| \Phi\left(\frac{|f|}{\frac{k}{|\lambda|}}\right) \right\|_{X}\leq 1 \right\} \\
& = &  |\lambda| \inf \left\{ \frac{k}{|\lambda|}>0 :  \left\| \Phi\left(\frac{|f|}{\frac{k}{|\lambda|}}\right) \right\|_{X}\leq 1 \right\} \ = \ |\lambda| \|f\|_{X^{\Phi}}.
\end{eqnarray*}

Now, let $f, g\in X^{\Phi}$ and take $K\geq 1$ as in (Q3). Given $a, b>0$ such that $\displaystyle\left\| \Phi\left(\frac{|f|}{a}\right) \right\|_{X}\leq 1$ and $\displaystyle\left\| \Phi\left(\frac{|g|}{b}\right) \right\|_{X}\leq 1,$ we have
\begin{eqnarray*}
\Phi\left(\frac{|f+g|}{K(a+b)}\right) & \leq & \frac{1}{K} \ \Phi\left(\frac{|f+g|}{a+b}\right) \ \leq \ \frac{1}{K} \ \Phi\left( \frac{a}{(a+b)}\frac{|f|}{a} + \frac{b}{(a+b)}\frac{|g|}{b}\right) \\
& \leq & \frac{1}{K} \frac{a}{(a+b)} \Phi\left(\frac{|f|}{a}\right) + \frac{1}{K} \frac{b}{(a+b)} \Phi\left(\frac{|g|}{b}\right)
\end{eqnarray*}
Hence,
$$
\left\| \Phi\left(\frac{|f+g|}{K(a+b)}\right) \right\|_{X} \leq \frac{a}{(a+b)} \left\|\Phi\left(\frac{|f|}{a}\right)\right\|_{X} + \frac{b}{(a+b)} \left\|\Phi\left(\frac{|g|}{b}\right)\right\|_{X} \leq 1
$$
which implies that $\|f+g\|_{X^{\Phi}} \leq K(a+b).$ By the arbitrariness of $a$ and $b$ we deduce that $\|f+g\|_{X^{\Phi}} \leq K(\|f\|_{X^{\Phi}}+\|g\|_{X^{\Phi}}).$ 

Thus, we have proved that $\|\cdot\|_{X^{\Phi}}$ is a quasi-norm in $X^{\Phi}$ with the same quasi-triangle constant as the one of the quasi-norm of $X.$ Moreover, we have already proved that $X^{\Phi}$ equipped with the Luxemburg quasi-norm is a quasi-normed space and an ideal in $L^0(\mu).$ It is also clear that the Luxemburg quasi-norm is a lattice quasi-norm: $|f|\leq |g|$ implies that $\displaystyle\Phi\left(\frac{|f|}{k}\right) \leq  \Phi\left(\frac{|g|}{k}\right)$ for all $k>0$  and this guarantees that $\|f\|_{X^{\Phi}} \leq \|g\|_{X^{\Phi}}.$ In addition, $\chi_{\Omega}\in X^{\Phi},$ since $\displaystyle\Phi\left(\frac{|\chi_{\Omega}|}{c}\right) = \Phi\left(\frac{1}{c}\right)\chi_{\Omega}\in X,$ for all $c>0,$ and hence $X^{\Phi}$ is in fact a quasi-normed function space.
\end{proof}

\begin{remark} 
\label{cont inclusion}
The inclusion of $X^{\Phi}\subseteq X$ is continuous provided $X$ and $X^{\Phi}$ be q-B.f.s. We will see in Theorem \ref{OS completeness} that the completeness is transferred from $X$ to $X^{\Phi}.$
\end{remark}

Once we have checked that $X^{\Phi}$ is a quasi-normed function space, it is immediate that $L^{\infty}(\mu)$ is contained in $X^{\Phi}$ and this inclusion is continuous with norm $\|\chi_{\Omega}\|_{X^{\Phi}}.$ The next result establishes the relation between the norm of this inclusion and the norm $\|\chi_{\Omega}\|_{X}$ of the continuous inclusion of  $L^{\infty}(\mu)$ into $X.$

\begin{lemma} 
\label{Linfinty inclusion in OS}
Let $\Phi$ be a Young function and $X$ be a quasi-normed function space over $\mu.$
\begin{itemize}
\item[(i)] For all $A\in\Sigma$ with $\mu(A)>0,$ $\|\chi_A\|_{X^{\Phi}} = \displaystyle\frac{1}{\Phi^{-1}\left( \frac{1}{\|\chi_A\|_X} \right)}.$
\item[(ii)] For all $f\in L^{\infty}(\mu),$ $\|f\|_{X^{\Phi}} \leq \displaystyle\frac{\|f\|_{L^{\infty}(\mu)}}{\Phi^{-1}\left( \frac{1}{\|\chi_{\Omega}\|_X} \right)}.$
\end{itemize}
\end{lemma}

\begin{proof}
(i) Write $\alpha:=\displaystyle\frac{1}{\Phi^{-1}\left( \frac{1}{\|\chi_A\|_X} \right)}.$ On the one hand,
$$
\left\| \Phi\left(\frac{|\chi_A|}{\alpha}\right) \right\|_{X} = \Phi\left(\frac{1}{\alpha}\right) \|\chi_A\|_{X} = \Phi\left(\Phi^{-1}\left( \frac{1}{\|\chi_A\|_X} \right) \right) \|\chi_A\|_{X} = 1,
$$
and so $\|\chi_A\|_{X^{\Phi}}\leq \alpha.$ On the other hand, given $k>0$ such that $\displaystyle\frac{\chi_A}{k}\in\widetilde{X}^{\Phi}$ with $\displaystyle\left\| \Phi\left(\frac{\chi_A}{k}\right) \right\|_{X}\leq 1,$ we have $\displaystyle\Phi\left(\frac{1}{k}\right) \|\chi_A\|_{X}\leq 1,$ that is, $\displaystyle\Phi\left(\frac{1}{k}\right) \leq \frac{1}{\|\chi_A\|_{X}}$ or, equivalently, $\displaystyle\frac{1}{k} \leq \Phi^{-1}\left(\frac{1}{\|\chi_A\|_{X}}\right),$ which finally leads to $\alpha\leq k$ and so $\alpha\leq \|\chi_A\|_{X^{\Phi}}.$

(ii) Since $|f|\leq \|f\|_{L^{\infty}(\mu)}\chi_{\Omega},$ for any $f\in L^{\infty}(\mu),$ we have  
$$
\|f\|_{X^{\Phi}} \leq \|f\|_{L^{\infty}(\mu)} \|\chi_{\Omega}\|_{X^{\Phi}}
$$ 
and the result follows applying (i) to $\chi_{\Omega}.$
\end{proof}

The following two results explore the close relationship between the quantities $\|f\|_{X^{\Phi}}$ and $\|\Phi(|f|)\|_X.$ This entails interesting consequences on boundedness in $X^{\Phi},$ allowing us to obtain a sufficient condition and a necessary condition for it.

\begin{lemma} 
\label{OS boundedness}
Let $\Phi$ be a Young function, $X$ be a quasi-normed function space over $\mu$ and $H\subset L^0(\mu).$
\begin{itemize}
\item[(i)] If $f\in \widetilde{X}^{\Phi},$ then $\|f\|_{X^{\Phi}} \leq \max\{1, \|\Phi(|f|)\|_X \}.$
\item[(ii)] If $\{ \Phi(|h|) : h\in H \}$ is bounded in $X,$ then $H$ is bounded in $X^{\Phi}.$
\end{itemize}
\end{lemma}

\begin{proof}
(i) On the one hand, $\|\Phi(|f|)\|_X\leq 1$ directly implies that  
$$
\|f\|_{X^{\Phi}} \leq 1 = \max\{1, \|\Phi(|f|)\|_X \}.
$$
On the other hand, if $\|\Phi(|f|)\|_X\geq 1,$ then
$$
\displaystyle\Phi\left(\frac{|f|}{\|\Phi(|f|)\|_X}\right) \leq \displaystyle\frac{1}{\|\Phi(|f|)\|_X} \Phi(|f|)\in X
$$
and hence $\displaystyle\Phi\left(\frac{|f|}{\|\Phi(|f|)\|_X}\right)\in X$ with $\displaystyle\left\|\Phi\left(\frac{|f|}{\|\Phi(|f|)\|_X}\right)\right\|_{X} \leq 1.$ This also leads to $\|f\|_{X^{\Phi}} \leq  \|\Phi(|f|)\|_X = \max\{1, \|\Phi(|f|)\|_X \}.$

(ii) If $\|\Phi(|h|)\|_X \leq M<\infty,$ for all $h\in H,$ according to (i) we have that  $\|h\|_{X^{\Phi}} \leq \max\{1, \|\Phi(|h|)\|_X \}\leq \max\{1,M\}<\infty,$ for all $h\in H.$
\end{proof}

\begin{lemma} 
\label{LQN-Modular relations}
Let $\Phi$ be a Young function, $X$ be a quasi-normed function space over $\mu$ and $f\in X^{\Phi}.$
\begin{itemize}
\item[(i)] If $\|f\|_{X^{\Phi}}<1,$ then $f\in\widetilde{X}^{\Phi}$ with $\|\Phi(|f|)\|_X \leq \|f\|_{X^{\Phi}}.$
\item[(ii)] If $\|f\|_{X^{\Phi}}>1$ and $f\in\widetilde{X}^{\Phi},$ then $\|\Phi(|f|)\|_X \geq \|f\|_{X^{\Phi}}.$
\item[(iii)] If $H\subseteq X^{\Phi}$ is bounded, then there exists a Young function $\Psi$ such that the set $\{ \Psi(|h|) : h\in H \}$ is bounded in $X.$
\end{itemize}
\end{lemma}

\begin{proof}
(i) Given $0<k<1$ such that $\displaystyle\frac{|f|}{k}\in\widetilde{X}^{\Phi}$ with $\displaystyle\left\| \Phi\left(\frac{|f|}{k}\right) \right\|_{X}\leq 1,$ we have
$\displaystyle
\Phi(|f|) = \Phi\left(k\frac{|f|}{k}\right) \leq k \ \Phi\left(\frac{|f|}{k}\right)\in X.
$
Therefore, $\Phi(|f|)\in X$ with $\|\Phi(|f|)\|_X \leq k \displaystyle\left\| \Phi\left(\frac{|f|}{k}\right) \right\|_{X}\leq k$ and keeping in mind that $\|f\|_{X^{\Phi}}<1,$ we obtain
$$
\|\Phi(|f|)\|_X  \leq \inf \left\{ 0<k<1 : \frac{|f|}{k}\in\widetilde{X}^{\Phi} \mbox{ with } \left\| \Phi\left(\frac{|f|}{k}\right) \right\|_{X}\leq 1 \right\} = \|f\|_{X^{\Phi}}.
$$

(ii) Let $0<\varepsilon<\|f\|_{X^{\Phi}}-1$ and observe that $\displaystyle\left\| \Phi\left(\frac{|f|}{\|f\|_{X^{\Phi}}-\varepsilon}\right) \right\|_{X} > 1.$ Thus,
\begin{eqnarray*}
\|\Phi(|f|)\|_X & = & \left\|\Phi\left((\|f\|_{X^{\Phi}}-\varepsilon)\frac{|f|}{\|f\|_{X^{\Phi}}-\varepsilon}\right)\right\|_X \\
& \geq & (\|f\|_{X^{\Phi}}-\varepsilon) \left\| \Phi\left(\frac{|f|}{\|f\|_{X^{\Phi}}-\varepsilon}\right) \right\|_{X} \geq \|f\|_{X^{\Phi}}-\varepsilon,
\end{eqnarray*}
and letting $\varepsilon\to 0,$ it follows that $\|\Phi(|f|)\|_X \geq \|f\|_{X^{\Phi}}.$

\noindent
(iii) Take $M>0$ such that $\|h\|_{X^{\Phi}}<M,$ for all $h\in H.$ Since $\displaystyle\left\|\frac{h}{M}\right\|_{X^{\Phi}}< 1,$ for all $h\in H,$ statement (i) guarantees that $\displaystyle\Phi\left(\frac{|h|}{M}\right)\in X$ with $\displaystyle\left\| \Phi\left(\frac{|h|}{M}\right) \right\|_{X}\leq \left\|\frac{h}{M}\right\|_{X^{\Phi}}<1,$ for all $h\in H.$ Defining $\Psi(t):=\Phi\displaystyle\left(\frac{t}{M}\right),$ for all $t\geq 0,$ we produce a Young function such that $\{ \Psi(|h|) : h\in H \}$ is bounded in $X.$
\end{proof}

We are now in a position to establish the remarkable fact that Orlicz spaces $X^{\Phi}$ are always complete for any q-B.f.s. $X.$ It is worth pointing out that standard proofs in the Banach setting require the $\sigma$-Fatou property of $X$ to obtain the $\sigma$-Fatou property of $X^{\Phi}$ (see the next Theorem \ref{OS Fatou}) and as a byproduct, the completeness of this last space. However, as we have said before, there are many complete spaces without the $\sigma$-Fatou property, to which it is not possible to apply the Theorem \ref{OS Fatou}. Herein lies the importance of the result that we will show next about completeness of $X^\Phi.$

\begin{theorem} 
\label{OS completeness}
Let $\Phi$ a Young function and $X$ be a q-B.f.s. over $\mu.$ Then, $X^{\Phi}$ is complete (and hence it is a q-B.f.s. over $\mu$).
\end{theorem}

\begin{proof}
Let $(h_n)_n$ be a positive increasing Cauchy sequence in $X^{\Phi}$ and take $K\geq 1$ as in (Q3). Then, we can choose a subsequence of $(h_n)_n,$ that we denote by $(f_n)_n,$ such that $\|f_{n+1}-f_{n}\|_{X^{\Phi}} < \displaystyle\frac{1}{2^{2n} K^{2n}},$ for all $n\in\mathbb{N}.$ Thus, 
$$
\left\|2^{n} K^{n} (f_{n+1}-f_{n}) \right\|_{X^{\Phi}} < \displaystyle\frac{1}{2^{n} K^{n}}< 1
$$ 
for all $n\in\mathbb{N},$ and by Lemma \ref{LQN-Modular relations} it follows that
$$
\left\|\Phi\left(2^{n} K^{n} \left(f_{n+1}-f_{n}\right)\right)\right\|_{X} \leq \left\|2^{n} K^{n} \left(f_{n+1}-f_{n}\right) \right\|_{X^{\Phi}} < \frac{1}{2^{n} K^{n}}, \quad n\in\mathbb{N},
$$
which proves that
$\displaystyle
\sum_{n=1}^{\infty} K^n \left\|\Phi\left(2^{n} K^{n} \left(f_{n+1}-f_{n}\right) \right)\right\|_{X} \leq \sum_{n=1}^{\infty} \frac{1}{2^{n}} < \infty.
$
The completeness of $X$ ensures that $f:=\displaystyle\sum_{n=1}^{\infty} \Phi\left(2^{n} K^{n}  \left(f_{n+1}-f_{n}\right) \right) \in X,$ by Theorem \ref{Maligranda completeness}. Note that $f\in L^0(\mu)$ and the convergence of that series is also $\mu$-a.e, since $X$ is continuously included in $L^0(\mu).$ Given $N\in\mathbb{N},$ let $g_N:=\displaystyle\sum_{n=1}^{N} (f_{n+1}-f_{n})$ and denote by $g:=\displaystyle\sup_{N} g_N $ pointwise $\mu$-a.e. Applying (\ref{quasi-subaditivity}) with $\alpha:=K,$ it follows that for all $N\in\mathbb{N},$
\begin{eqnarray*}
\Phi(g_N)  & = & \Phi\left(\sum_{n=1}^{N} (f_{n+1}-f_{n}) \right) \leq \sum_{n=1}^{N} \frac{1}{2^n K^n} \Phi\left(2^n K^n \left(f_{n+1} -f_{n}\right) \right)
\\
& \leq & \sum_{n=1}^{N} \Phi(2^{n} K^{n} (f_{n+1}-f_{n}) ) \leq f
\end{eqnarray*}
Therefore, $0\leq g_N\leq \Phi^{-1}(f)\in L^0(\mu)$ for all $N\in\mathbb{N}$ and so $g\in L^0(\mu)$ with $0\leq g\leq \Phi^{-1}(f) \in X^{\Phi},$ which guarantees that $g\in X^{\Phi}.$ But
$$
f_{N+1} =\sum_{n=1}^{N} (f_{n+1}-f_{n}) + f_{1} = g_{N} + f_1
$$
for all $N\in\mathbb{N}$ and so there also exists $\displaystyle\sup_{n} f_n = g + f_1 \in X^{\Phi}.$ Since $(f_n)_n$ is a subsequence of the original increasing sequence $(h_n)_n,$ the supremum of the whole sequence must exists and be the same as the supremum of the subsequence. By applying Amemiya's Theorem \ref{Amemiya completeness} we conclude that $X^{\Phi}$ is complete.
\end{proof}

If the q-B.f.s. $X$ has the $\sigma$-Fatou property, then we can improve a little more our knowledge about $X^{\Phi}$ as the following proposition makes evident.

\begin{theorem} 
\label{OS Fatou}
Let $\Phi$ be a Young function and $X$ be a q-B.f.s. over $\mu$ with the $\sigma$-Fatou property.
\begin{itemize}
\item[(i)] If $0\neq f\in X^{\Phi}$ then $\displaystyle\frac{|f|}{\|f\|_{X^{\Phi}}}\in\widetilde{X}^{\Phi}$ with $\displaystyle\left\| \Phi\left(\frac{|f|}{\|f\|_{X^{\Phi}}}\right) \right\|_{X}\leq 1.$
\item[(ii)] If $f\in X^{\Phi}$ with $\|f\|_{X^{\Phi}}\leq 1$ then $f\in\widetilde{X}^{\Phi}$ with $\|\Phi(|f|)\|_X \leq \|f\|_{X^{\Phi}}.$
\item[(iii)] $X^{\Phi}$ also has the $\sigma$-Fatou property.
\end{itemize}
\end{theorem}

\begin{proof}
(i) Take a sequence $(k_n)_n$ such that $k_n \downarrow \|f\|_{X^{\Phi}}$ and $\displaystyle\left\| \Phi\left(\frac{|f|}{k_n}\right) \right\|_{X}\leq 1,$ for all $n\in\mathbb{N}.$ Then, $\displaystyle\frac{|f|}{k_n}\uparrow \frac{|f|}{\|f\|_{X^{\Phi}}}$ and so $\displaystyle\Phi\left(\frac{|f|}{k_n}\right)\uparrow \Phi\left(\frac{|f|}{\|f\|_{X^{\Phi}}}\right),$ since $\Phi$ is continuous and increasing. The $\sigma$-Fatou property of $X$ guarantees that $\displaystyle\Phi\left(\frac{|f|}{\|f\|_{X^{\Phi}}}\right)\in X$ and $\displaystyle\left\| \Phi\left(\frac{|f|}{\|f\|_{X^{\Phi}}}\right) \right\|_{X} = \sup_{n} \left\| \Phi\left(\frac{|f|}{k_n}\right) \right\|_X \leq 1.$

(ii) According to (i) and the inequality
$$
\Phi(|f|) = \Phi\left(\|f\|_{X^{\Phi}}\frac{|f|}{\|f\|_{X^{\Phi}}}\right) \leq \|f\|_{X^{\Phi}} \ \Phi\left(\frac{|f|}{\|f\|_{X^{\Phi}}}\right)
$$
we deduce that $\Phi(|f|)\in X$ and
$$
\|\Phi(|f|)\|_X \leq \|f\|_{X^{\Phi}} \left\| \Phi\left(\frac{|f|}{\|f\|_{X^{\Phi}}}\right) \right\|_{X}\leq \|f\|_{X^{\Phi}}.
$$

(iii) Let $(f_n)_n$ in $X^{\Phi}$ with $0\leq f_n\uparrow f$ $\mu$-a.e. and $M:=\displaystyle\sup_n \|f_n\|_{X^{\Phi}}<\infty.$ Then, $\displaystyle\Phi\left(\frac{f_n}{M}\right)\uparrow \Phi\left(\frac{f}{M}\right)$ $\mu$-a.e. and $\displaystyle\left\|\frac{f_n}{M}\right\|_{X^{\Phi}}\leq 1$ for all $n\in\mathbb{N}.$ Applying (ii), we deduce that $\displaystyle\Phi\left(\frac{f_n}{M}\right)\in X$ with $\displaystyle\left\|\Phi\left(\frac{f_n}{M}\right)\right\|_{X}\leq 1$ for all $n\in\mathbb{N}$ and using the $\sigma$-Fatou property of $X,$ it follows that $\displaystyle\Phi\left(\frac{f}{M}\right)\in X$ with 
$$
\left\|\Phi\left(\frac{f}{M}\right)\right\|_{X} = \sup_n \left\| \Phi\left(\frac{f_n}{M}\right) \right\|_{X}\leq 1.
$$
This implies that $f\in X^{\Phi}$ with $\|f\|_{X^{\Phi}}\leq M$ and we also have $M\leq \|f\|_{X^{\Phi}},$ since $f_n\leq f\in X^{\Phi}.$ Thus, $\|f\|_{X^{\Phi}}=M,$ which proves that $X^{\Phi}$ has the $\sigma$-Fatou property.
\end{proof}

The relation between the Orlicz class and its corresponding Orlicz space is greatly simplified when the Young function has the $\Delta_2$-property. In addition, this has far-reaching consequences on convergence in $X^{\Phi}$ as we state in the next result.

\begin{theorem} 
\label{OS Delta2}
Let $X$ be a quasi-normed function space over $\mu$ and $\Phi\in\Delta_2.$
\begin{itemize}
\item[(i)] The Orlicz space and the Orlicz class coincide: $X^{\Phi} = \widetilde{X}^{\Phi}.$
\item[(ii)] $\|f_n\|_{X^{\Phi}}\to 0$ if and only if $\|\Phi(|f_n|)\|_{X}\to 0,$ for all $(f_n)\subseteq X^{\Phi}.$
\item[(iii)] If $X$ is $\sigma$-order continuous, then $X^{\Phi}$ is also $\sigma$-order continuous.
\end{itemize}
\end{theorem}

\begin{proof}
(i) Given $f\in X^{\Phi},$ there exists $c>0$ such that $\displaystyle\Phi\left(\frac{|f|}{c}\right)\in X.$ If $c\leq 1,$ then
$\displaystyle
\Phi(|f|) = \Phi\left(c \ \frac{|f|}{c}\right) \leq c \ \Phi\left(\frac{|f|}{c}\right) \in X,
$
and if $c>1,$ then there exist $C>1$ such that $\Phi(ct)\leq C\Phi(t)$ for all $t\geq 0$ by the $\Delta_2$-property of $\Phi.$ Therefore,
$\displaystyle
\Phi(|f|) = \Phi\left(c \ \frac{|f|}{c}\right) \leq C \ \Phi\left(\frac{|f|}{c}\right) \in X.
$
In any case, it follows that $\Phi(|f|)\in X,$ which means that $f\in \widetilde{X}^{\Phi}.$

\noindent
(ii) If $\|f_n\|_{X^{\Phi}}\to 0,$ then $\|\Phi(|f_n|)\|_{X}\to 0$ as a consequence of Lemma \ref{LQN-Modular relations} (i). Suppose now that
 $\|f_n\|_{X^{\Phi}}$ does not converges to $0.$ Then, there exists $\varepsilon>0$ and a subsequence $(f_{n_k})$ of $(f_n)$ such that $\|f_{n_k}\|_{X^{\Phi}}>\varepsilon$ for all $k\in\mathbb{N}.$ We can assume that $\varepsilon<1$ and that $(f_{n_k})$ is the whole $(f_n)$ without loss of generality. Since $\Phi\in\Delta_2$ and $\displaystyle\frac{1}{\varepsilon}>1,$ there exist $C>1$ such that $\displaystyle\Phi\left(\frac{|f_n|}{\varepsilon}\right) \leq C \Phi(|f_n|).$ By (i), we deduce that $\displaystyle\Phi\left(\frac{|f_n|}{\varepsilon}\right)\in X$ and hence $\displaystyle\left\| \Phi\left(\frac{|f_n|}{\varepsilon}\right) \right\|_{X} > 1.$ Thus, 
 $$
 \|\Phi(|f_n|)\|_{X} \geq \displaystyle\frac{1}{C} \displaystyle\left\| \Phi\left(\frac{|f_n|}{\varepsilon}\right) \right\|_{X} > \frac{1}{C} >0,
 $$
 which means that $\|\Phi(|f_n|)\|_{X}$ does not converges to $0.$

\noindent
(iii) Let $(f_n)_n$ and $f$ in $X^{\Phi}$ such that $0\leq f_n\uparrow f$ $\mu$-a.e. Then, $\Phi\left(f-f_n\right)\downarrow 0$ $\mu$-a.e. Since $X$ is $\sigma$-order continuous, it follows that $\|\Phi\left(f-f_n\right)\|_{X}\to 0$ and by (ii) this implies that $\|f-f_n\|_{X^{\Phi}}\to 0,$ which gives the $\sigma$-order continuity of $X^{\Phi}.$
\end{proof}

\section{Applications: Orlicz spaces associated to a vector measure}

First of all observe that classical Orlicz spaces $L^{\Phi}(\mu)$ with respect to a po\-sitive finite measure $\mu$ are obtained applying the construction $X^{\Phi}$ of section 4 to the B.f.s. $X=L^1(\mu),$ that is, $L^{\Phi}(\mu) = L^1(\mu)^{\Phi}$ equipped with the norm $\|\cdot\|_{L^{\Phi}(\mu)} := \|\cdot\|_{L^1(\mu)^{\Phi}}.$ Using these classical Orlicz spaces, the Orlicz spaces $L^{\Phi}_w(m)$ and $L^{\Phi}(m)$ with respect to a vector measure $m:\Sigma\to Y$ were introduced in \cite{Del} in the following way:
$$
L^{\Phi}_w(m) := \left\{ f\in L^0(m) : f\in L^{\Phi}(|\langle m, y^*\rangle|), \ \forall\, y^*\in Y^* \right\},
$$
equipped with the norm
$$
\|f\|_{L^{\Phi}_w(m)} := \sup \left\{ \|f\|_{L^{\Phi}(|\langle m, y^*\rangle|)} : y^*\in B_{Y^*}\right\},
$$
and $L^{\Phi}(m)$ is the closure of simple functions $\mathscr{S}(\Sigma)$ in $L^{\Phi}_w(m).$ The next result establishes that these Orlicz spaces $L^{\Phi}_w(m)$ and $L^{\Phi}(m)$ can be obtained as generalized Orlicz spaces $X^{\Phi}$ by taking $X$ to be $L^1_w(m)$ and $L^1(m),$ respectively.

\begin{proposition}
Let $\Phi$ be a Young function and $m:\Sigma\to Y$ a vector measure.
\begin{itemize}
\item[(i)] $L^{\Phi}_w(m) = L^1_w(m)^{\Phi}$ and $\|f\|_{L^{\Phi}_w(m)} = \|f\|_{L^1_w(m)^{\Phi}},$ for all $f\in L^{\Phi}_w(m).$
\item[(ii)] $L^{\Phi}(m) \subseteq L^1(m)^{\Phi}$ and if $\Phi\in\Delta_2,$ then $L^{\Phi}(m) = L^1(m)^{\Phi}.$
\end{itemize}
\end{proposition}

\begin{proof}
(i) Suppose that $f\in L^1_w(m)^{\Phi}$ and let  $k>0$ such that $\displaystyle\Phi\left(\frac{|f|}{k}\right)\in L^1_w(m)$ with $\displaystyle\left\| \Phi\left(\frac{|f|}{k}\right) \right\|_{L^1_w(m)}\leq 1.$ Given $y^*\in B_{Y^*}$ we have $\displaystyle\Phi\left(\frac{|f|}{k}\right)\in L^1(|\langle m, y^*\rangle|)$ with $\displaystyle\left\| \Phi\left(\frac{|f|}{k}\right) \right\|_{L^1(|\langle m, y^*\rangle|)} \leq \displaystyle\left\| \Phi\left(\frac{|f|}{k}\right) \right\|_{L^1_w(m)}\leq 1.$ This implies that $f\in L^{\Phi}(|\langle m, y^*\rangle|)$ with $\|f\|_{L^{\Phi}(|\langle m, y^*\rangle|)}\leq k.$ Hence, $f\in L^{\Phi}_w(m)$ with $\|f\|_{L^{\Phi}_w(m)} \leq \|f\|_{L^1_w(m)^{\Phi}}.$

Reciprocally, suppose now that $f\in L^{\Phi}_w(m),$ write $M:=\|f\|_{L^{\Phi}_w(m)}$ and let $y^*\in B_{Y^*}.$ Since $f\in L^{\Phi}(|\langle m, y^*\rangle|)$ and $\|f\|_{L^{\Phi}(|\langle m, y^*\rangle|)} \leq M,$ we have that $\displaystyle\frac{f}{M}\in L^{\Phi}(|\langle m, y^*\rangle|)$ with  
$\displaystyle
\left\|\frac{f}{M}\right\|_{L^{\Phi}(|\langle m, y^*\rangle|)} \leq 1.$ 
Applying Theorem \ref{OS Fatou} (ii) to the space $X=L^{1}(|\langle m, y^*\rangle|),$ it follows that $\displaystyle\Phi\left(\frac{|f|}{M}\right)\in L^1(|\langle m, y^*\rangle|)$ with 
$$
\left\| \Phi\left(\frac{|f|}{M}\right) \right\|_{L^1(|\langle m, y^*\rangle|)} \leq \displaystyle\left\|\frac{f}{M}\right\|_{L^{\Phi}(|\langle m, y^*\rangle|)} \leq 1.
$$ 
Then, the arbitrariness of $y^*\in B_{Y^*}$ guarantees that $\displaystyle\Phi\left(\frac{|f|}{M}\right)\in L^1_w(m)$ with  $\displaystyle\left\| \Phi\left(\frac{|f|}{M}\right) \right\|_{L^1_w(m)} \leq 1$ and hence $f\in L^1_w(m)^{\Phi}$ with $\|f\|_{L^1_w(m)^{\Phi}} \leq M.$

(ii) Since $L^1(m)^{\Phi}$ is a B.f.s., simple functions $\mathscr{S}(\Sigma) \subseteq L^1(m)^{\Phi}$ and $L^1(m)^{\Phi}$ is a closed subspace of $L^1_w(m)^{\Phi}.$ Thus, taking in account (i), we deduce that $L^{\Phi}(m) \subseteq L^1(m)^{\Phi}.$ If in addition $\Phi\in\Delta_2,$ we have
$$
L^1(m)^{\Phi} = \{ f\in L^0(m) : \Phi(|f|)\in L^1(m) \} = L^{\Phi}(m),
$$
where the first equality is due to Theorem \ref{OS Delta2} (i) applied to $X=L^1(m)$ and the second one can be found in \cite[Proposition 4.4]{Del}.
\end{proof}

The Orlicz spaces $L^{\Phi}(m)$ have been recently employed in \cite{CFMN2} to locate the compact subsets of $L^{1}(m).$ Motivated by the idea of studying compactness in $L^{1}(\|m\|)$ in a forthcoming paper \cite{preprint}, we introduce the Orlicz spaces $L^{\Phi}(\|m\|)$ as the Orlicz spaces $X^{\Phi}$ associated to the q-B.f.s. $X = L^1(\|m\|).$ For further reference, we collect together all the information that our general theory provide about these new Orlicz spaces.

\begin{definition}
Let $\Phi$ be a Young function and $m:\Sigma\to Y$ a vector measure. We define the {\em Orlicz spaces associated to the semivariation\/} of $m$ as $L^{\Phi}(\|m\|):=L^1(\|m\|)^{\Phi}$ equipped with $\|f\|_{L^{\Phi}(\|m\|)}:= \|f\|_{L^1(\|m\|)^{\Phi}},$ for all $f\in L^{\Phi}(\|m\|).$
\end{definition}

\begin{corollary}
Let $\Phi$ be a Young function, $m:\Sigma\to Y$ a vector measure and $\mu$ any Rybakov control measure for $m.$ Then,
\begin{itemize}
\item[(i)] $L^{\Phi}(\|m\|)$ is a q-B.f.s. over $\mu$ with the $\sigma$-Fatou property.
\item[(ii)] If $\Phi\in\Delta_2,$ then $L^{\Phi}(\|m\|)$ is $\sigma$-order continuous.
\item[(iii)] $L^{\Phi}(\|m\|) \subseteq L^1(\|m\|)$ with continuous inclusion.
\end{itemize}
\end{corollary}

\begin{proof}
Apply Theorems \ref{LQN properties}, \ref{OS Fatou} and \ref{OS Delta2} to the q-B.f.s $X=L^1(\|m\|).$ See also Proposition~\ref{OS properties} and Remark \ref{cont inclusion}.
\end{proof}

\begin{corollary}
Let $\Phi$ be a Young function, $m:\Sigma\to Y$ a vector measure, $f\in L^{\Phi}(\|m\|)$ and $H\subseteq L^0(m).$
\begin{itemize}
\item[(i)] If $\Phi(|f|)\in L^1(\|m\|),$ then $\|f\|_{L^{\Phi}(\|m\|)} \leq \max\{1, \|\Phi(|f|)\|_{L^1(\|m\|)} \}.$
\item[(ii)] If $\|f\|_{L^{\Phi}(\|m\|)}\leq 1,$ then $\Phi(|f|)\in L^1(\|m\|)$ and $\|\Phi(|f|)\|_{L^1(\|m\|)} \leq \|f\|_{L^{\Phi}(\|m\|)}.$
\item[(iii)] If $\|f\|_{L^{\Phi}(\|m\|)}>1$ and $\Phi(|f|)\in L^1(\|m\|),$ then $\|\Phi(|f|)\|_{L^1(\|m\|)} \geq \|f\|_{L^{\Phi}(\|m\|)}.$
\item[(iv)] If $\{\Phi(|h|) : h\in H \}$ is bounded in $L^{1}(\|m\|),$ then $H$ is bounded in $L^{\Phi}(\|m\|).$
\item[(v)] If $H$ is bounded in $L^{\Phi}(\|m\|),$ then there exists a Young function $\Psi$ such that $\{\Psi(|h|) : h\in H \}$ is bounded in $L^{1}(\|m\|).$
\end{itemize}
\end{corollary}

\begin{proof}
Particularize Lemmas \ref{OS boundedness} and \ref{LQN-Modular relations} to $X=L^1(\|m\|).$ Note that, in fact, we can use (ii) of Theorem \ref{OS Fatou}.
\end{proof}

\begin{corollary}
Let $\Phi\in\Delta_2,$ $m:\Sigma\to Y$ a vector measure and $(f_n)_n\subseteq L^{\Phi}(\|m\|).$
\begin{itemize}
\item[(i)] $L^{\Phi}(\|m\|) = \{ f\in L^0(m) : \Phi(|f|)\in L^1(\|m\|) \}.$
\item[(ii)] $\|f_n\|_{L^{\Phi}(\|m\|)}\to 0$ if and only if $\|\Phi(|f_n|)\|_{L^1(\|m\|)}\to 0.$
\end{itemize}
\end{corollary}

\begin{proof}
Apply Theorem \ref{OS Delta2} to the space $X=L^1(\|m\|).$
\end{proof}

\section{Applications: Interpolation of Orlicz spaces}

In this section all the q-B.f.s. will be supposed to be complex. This means that $L^0(\mu)$ will be assumed to be in fact the space of all ($\mu$-a.e. equivalence classes of) $\mathbb{C}$-valued measurable functions on $\Omega.$ Recall that a complex q-B.f.s $X$ over $\mu$ is the {\em complexification\/} of the real q-B.f.s. $X_{\mathbb{R}}:=X\cap L^0_{\mathbb{R}}(\mu),$ where $L^0_{\mathbb{R}}(\mu)$ is the space of all ($\mu$-a.e. equivalence classes of) $\mathbb{R}$-valued measurable functions on $\Omega$ (see \cite[p.24]{ORS} for more details) and this allows to extend all the real q-B.f.s. defined above to complex q-B.f.s. following a standard argument.

The complex method of interpolation, $[X_0,X_1]_{\theta}$ with $0<\theta<1,$ for pairs $(X_0,X_1)$ of quasi-Banach spaces was introduced in \cite{KaMi} as a natural extension of Calder\'on's original definition for Banach spaces. It relies on a theory of analytic functions with values in quasi-Banach spaces which was developed in \cite{Kal1} and \cite{Kal2}. It is important to note that there is no analogue of the Maximum Modulus Principle for general quasi-Banach spaces, but there is a wide subclass of quasi-Banach spaces called {\em analytically convex\/} (A-convex) in which that principle does hold. For a q-B.f.s. $X$ it can be proved that analytical convexity is equivalent to {\em lattice convexity\/} (L-convexity), i.e., there exists $0<\varepsilon<1$ so that if $f\in X$ and $0\leq f_i\leq f,$ $i=1,\dots,n,$ satisfy $\displaystyle\frac{f_1+\cdots + f_n}{n}\geq (1-\varepsilon)f,$ then $\displaystyle\max_{1\leq i \leq n}\|f_i\|_{X} \geq \varepsilon \|f\|_{X}$ (see \cite[Theorem 4.4]{Kal2}). This is also equivalent to $X$ be $s$-convex for some $s>0$ (see \cite[Theorem~2.2]{Kal0}). We recall that $X$ is called $s$-convex if there exists $C\geq 1$ such that
$$
\left\| \left( \sum_{k=1}^{n}|f_k|^s \right)^{\frac{1}{s}} \right\|_{X} \leq C \left( \sum_{k=1}^{n} \|f_k\|_{X}^{s} \right)^{\frac{1}{s}}
$$
for all $n\in\mathbb{N}$ and $f_1, \dots, f_n\in X.$ Observe that, $X$ is $s$-convex if and only if its $s$-th power $X_{[s]}$ is $1$-convex, where the {\em $s$-th power\/} $X_{[s]}$ of a q-B.f.s. $X$ over $\mu$ (for any $0<s<\infty$) is the q-B.f.s. $X_{[s]}:=\displaystyle\left\{ f\in L^0(\mu) : |f|^{\frac{1}{s}}\in X \right\}$ equipped with the quasi-norm $\|f\|_{X_{[s]}} = \displaystyle\left\| |f|^{\frac{1}{s}} \right\|_{X}^{s},$ for all $f\in X_{[s]}$ (see \cite[Proposition~2.22]{ORS}).

The following result provide a condition under which the L-convexity of $X$ can be transferred to its Orlicz space $X^{\Phi}.$ When $X$ possesses the $\sigma$-Fatou property, this can be derived from \cite[Proposition 3.3]{KMP}, but we make apparent that this property can be dropped.
Recall that a function $\psi$ on the semiaxis $[0,\infty)$ is said to be {\em quasiconcave\/} if $\psi(0)=0,$ $\psi(t)$ is positive and increasing for $t>0$ and $\displaystyle\frac{\psi(t)}{t}$ is decreasing for $t>0.$ Observe that a quasiconcave function $\psi$ satisfies the following inequalities for all $t\geq 0$:
$$
\left\{
\begin{array}{ccl}
\psi(\alpha t) \geq \alpha \, \psi(t) & \mbox{ if } & 0\leq\alpha\leq 1,
\\
\psi(\alpha t) \leq \alpha \, \psi(t) & \mbox{ if } & \alpha\geq 1.
\end{array}
\right.
$$

\begin{theorem} 
\label{OS L-convexity}
If $X$ is an L-convex q-B.f.s. and $\Phi\in\Delta_2,$ then $X^{\Phi}$ is L-convex.
\end{theorem}

\begin{proof}
Since $\Phi\in\Delta_2,$ there exists $s>1$ such that $\Phi(2t)\leq s\Phi(t)$ for all $t\geq 0.$ From the inequality
$$
t \Phi'(t)\leq \int_{t}^{2t} \Phi'(u)\, du \leq \int_{0}^{2t} \Phi'(u)\, du = \Phi(2t)\leq s\Phi(t), \ t>0
$$
it is easy to check that $\displaystyle\frac{\Phi(t)}{t^s}$ is decreasing and then $\displaystyle\frac{\Phi\left(t^{\frac{1}{s}}\right)}{t}$ so is. Therefore, the function $\psi(t):=\Phi\left(t^{\frac{1}{s}}\right)$ is quasiconcave. Take $0<\delta<1$ such that $(1-\delta)^s=1-\varepsilon,$ where $\varepsilon$ is the constant from the L-convexity of $X.$ Let $f\in X^{\Phi}$ and $0\leq f_i\leq f,$ $i=1,\dots,n$ satisfying $\displaystyle\frac{f_1+\cdots + f_n}{n}\geq (1-\delta)f.$ We can also assume that $\|f\|_{X^{\Phi}}=1$ without loss of generality. Note that this implies $\|\Phi(f)\|_{X} \geq 1.$ If we suppose, on the contrary, that $\|\Phi(f)\|_{X}<1$ and we take $0<k<1$ such that $\|\Phi(f)\|_{X} < k^s < 1,$ then
$$
\left\| \Phi\left(\frac{f}{k}\right) \right\|_X = \left\| \psi\left(  \frac{f^s}{k^s}  \right) \right\|_{X} \leq \frac{1}{k^s} \| \psi(f^s) \|_{X} = \frac{1}{k^s} \| \Phi(f) \|_{X} <1,
$$
and therefore $\|f\|_{X^{\Phi}}<k<1.$ Moreover, we have $0\leq\Phi(f_i)\leq\Phi(f)\in X$ and
\begin{eqnarray*}
\frac{\Phi(f_1)+\cdots +\Phi(f_n)}{n} & \geq & \Phi\left(\frac{f_1+\cdots + f_n}{n}\right) \ \geq \ \Phi( (1-\delta)f )
\\
& \geq &  (1-\delta)^s \psi(f^s) \ = \ (1-\delta)^s \Phi(f) \ = \ (1-\varepsilon) \Phi(f).
\end{eqnarray*}
Thus, the L-convexity of $X$ implies that $\displaystyle\max_{1\leq i \leq n}\|\Phi(f_i)\|_{X} \geq \varepsilon \|\Phi(f)\|_{X} \geq \varepsilon$ and hence $\displaystyle\max_{1\leq i \leq n}\|f_i\|_{X^{\Phi}} \geq \varepsilon>\delta$ by (i) of Lemma \ref{LQN-Modular relations}.
\end{proof}

The {\em Calder\'{o}n product\/} $X_0^{1-\theta}X_1^{\theta}$ of two q-B.f.s. $X_0$ and $X_1$ over $\mu$ is the q-B.f.s. of all functions $f\in L^0(\mu)$ such that there exist $f_0\in B_{X_0},$ $f_1\in B_{X_1}$ and $\lambda>0$ for which
\begin{equation} \label{defCL}
|f(w)| \leq \lambda |f_0(w)|^{1-\theta} |f_1(w)|^{\theta}, \quad w\in\Omega \ \ (\mu\mbox{-a.e.})
\end{equation}
endowed with the quasi-norm $\|f\|_{X_0^{1-\theta}X_1^{\theta}}= \inf \lambda,$ where the infimum is taken over all $\lambda$ satisfying (\ref{defCL}). The complex method gives the result predicted by the Calder\'{o}n product for nice pairs of q-B.f.s. (see \cite[Theorem 3.4]{KaMi}).

\begin{theorem}
\label{Complex Method = Calderon Product}
Let $\Omega$ be a Polish space and let $\mu$ be a finite Borel measure on $\Omega.$ Let $X_0,$ $X_1$ be a pair of $\sigma$-order continuous L-convex q-B.f.s. over $\mu.$ Then $X_0+X_1$ is L-convex and $[X_0,X_1]_{\theta} = X_0^{1-\theta}X_1^{\theta}$ with equivalence of quasi-norms.
\end{theorem}

On the other hand, it is easy to compute the Calder\'{o}n product of two Orlicz spaces associated to the same q-B.f.s:

\begin{proposition} 
\label{CP of Orlicz spaces}
Let $X$ be a q-B.f.s. over $\mu,$ $\Phi_0,$ $\Phi_1$ Young functions, $0< \theta <1$ and $\Phi$ such that $\Phi^{-1} := (\Phi_{0}^{-1})^{1-\theta} (\Phi_{1}^{-1})^{\theta}.$ Then $\left(X^{\Phi_0}\right)^{1-\theta}\left(X^{\Phi_1}\right)^{\theta} = X^{\Phi}.$ \end{proposition}

\begin{proof}
Given $f\in X^{\Phi},$ there exists $c>0$ such that $h:=\displaystyle\Phi\left(\frac{|f|}{c}\right)\in X$ and hence $f_0:=\Phi_0^{-1}(h)\in X^{\Phi_0}$ and $f_1:=\Phi_1^{-1}(h)\in X^{\Phi_1}.$ Taking $\alpha:=\max\{ \|f_0\|_{X^{\Phi_0}}, \|f_1\|_{X^{\Phi_1}} \},$ it follows that
\begin{eqnarray*}
|f| & = & c\, \Phi^{-1}(h) \ = \ c \, (\Phi_0^{-1}(h))^{1-\theta} (\Phi_1^{-1}(h))^{\theta} \ = \ c |f_0|^{1-\theta} |f_1|^{\theta}\\
& \leq & c\alpha \left(\frac{f_0}{\alpha}\right)^{1-\theta} \left(\frac{f_1}{\alpha}\right)^{\theta},
\end{eqnarray*}
which yields $f\in \left(X^{\Phi_0}\right)^{1-\theta}\left(X^{\Phi_1}\right)^{\theta}.$

Conversely, if $f\in \left(X^{\Phi_0}\right)^{1-\theta}\left(X^{\Phi_1}\right)^{\theta},$ then there exist $\lambda>0,$ $f_0\in X^{\Phi_0}$ and $f_1\in X^{\Phi_1}$ such that $|f|\leq \lambda |f_0|^{1-\theta} |f_1|^{\theta}.$ This implies the existence of $c>0$ such that $h_0 := \displaystyle\Phi_0\left(\frac{|f_0|}{c}\right)\in X$ and $h_1 := \displaystyle\Phi_1\left(\frac{|f_1|}{c}\right)\in X.$ Thus, taking $h:=h_0+h_1\in X,$ we deduce that
\begin{eqnarray*}
|f| & \leq & \lambda |f_0|^{1-\theta} |f_1|^{\theta} = \lambda c \left(\frac{|f_0|}{c}\right)^{1-\theta} \left(\frac{|f_1|}{c}\right)^{\theta} = \lambda c (\Phi_0^{-1}(h_0))^{1-\theta} (\Phi_1^{-1}(h_1))^{\theta}\\
 & \leq & \lambda c (\Phi_0^{-1}(h))^{1-\theta} (\Phi_1^{-1}(h))^{\theta} = \lambda c \Phi^{-1}(h) \in X^{\Phi},
\end{eqnarray*}
and hence $f\in X^{\Phi}.$
\end{proof}

Combining the three previous results, we obtain conditions under which the complex method applied to Orlicz spaces associated to a q-B.f.s. over $\mu$ keeps on producing an Orlicz space associated to the same q-B.f.s.

\begin{corollary} 
\label{CM of Orlicz spaces}
Let $\Omega$ be a Polish space and let $\mu$ be a finite Borel measure on $\Omega.$ Let $X$ be an L-convex, $\sigma$-order continuous q-B.f.s. over $\mu,$ $\Phi_0, \Phi_1\in\Delta_2,$ $0<\theta<1$ and $\Phi$ such that $\Phi^{-1} := (\Phi_{0}^{-1})^{1-\theta} (\Phi_{1}^{-1})^{\theta}.$ Then,  
$$
\left[X^{\Phi_0},X^{\Phi_1}\right]_{\theta} = X^{\Phi}.
$$
\end{corollary}

\begin{proof}
According to Theorems \ref{OS Delta2} and \ref{OS L-convexity}, the hypotheses guarantee that $X^{\Phi_0}$ and $X^{\Phi_1}$ are $L$-convex, $\sigma$-order continuous q-B.f.s. Therefore, the result follows by applying Theorem \ref{Complex Method = Calderon Product} and Proposition \ref{CP of Orlicz spaces}.
\end{proof}

Let us denote $L^s(\|m\|):=L^1(\|m\|)_{\left[\frac{1}{s}\right]},$ for $0<s<\infty$ and $m:\Sigma\to Y$ a vector measure. In \cite[Proposition 4.1]{CFMN1} we proved that if $s>1,$ then $L^s(\|m\|)$ is $r$-convex for every $r < s.$ In fact, this is true for all $0<s<\infty$ because if $0<s\leq 1$ and $r<s,$ then $\displaystyle \frac{s}{r}>1$ and hence $L^{\frac{s}{r}}(\|m\|)$ is $1$-convex, that is $L^s(\|m\|)_{[r]}$ is $1$-convex, which is equivalent to $L^s(\|m\|)$ be $r$-convex. This means that $L^s(\|m\|)$ is $L$-convex for all $0<s<\infty.$ In particular, $L^1(\|m\|)$ is $L$-convex and we can apply Corollary~\ref{CM of Orlicz spaces} to it.

\begin{corollary}
Let $\Omega$ be a Polish space and let $\mu$ be a Borel measure which is a Rybakov control measure for $m.$ Let $\Phi_0, \Phi_1\in\Delta_2,$ $0<\theta<1$ and $\Phi$ such that $\Phi^{-1} := (\Phi_{0}^{-1})^{1-\theta} (\Phi_{1}^{-1})^{\theta}.$ Then, $[L^{\Phi_0}(\|m\|),L^{\Phi_1}(\|m\|)]_{\theta} =  L^{\Phi}(\|m\|).$
\end{corollary}

Note that, for $p>1,$ $\displaystyle\frac{1}{p}$-th powers are an special case of Orlicz spaces, since $X_{\left[\frac{1}{p}\right]} = X^{\Phi_{[p]}},$ where $\Phi_{[p]}(t)=t^p.$ If we particularize the previous Corollary to these powers, then we obtain the interpolation result below for $L^p(\|m\|)$ spaces. In fact, this result is valid for all $0<p_0, p_1<\infty$ due to the fact that the Calder\'on product \emph{commutes} with powers for all indices.

\begin{corollary}
Let $\Omega$ be a Polish space and let $\mu$ be a Borel measure which is a Rybakov control measure for $m.$ Let $0<\theta<1$ and $0<p_0, p_1<\infty.$ Then $\left[L^{p_0}(\|m\|),L^{p_1}(\|m\|)\right]_{\theta} =  L^{p}(\|m\|),$ where $\displaystyle\frac{1}{p} = \frac{1-\theta}{p_0} + \frac{\theta}{p_1}.$
\end{corollary}


\begin{bibdiv}
\begin{biblist}

\bib{AlBu}{book}{
   author={Aliprantis, C.D.},
   author={Burkinshaw, O.},
   title={Locally solid Riesz spaces},
   note={Pure and Applied Mathematics, Vol. 76},
   year={1978},
}

\bib{CFMMN}{article}{
   author={del Campo, R.}
   author={Fern\'andez, A.}
   author={Manzano, A.}
   author={Mayoral, F.}
   author={Naranjo, F.}
   title={Complex interpolation of Orlicz spaces with respect to a vector measure},
   journal={Math. Nachr.},
   volume={287},
   number={1},
   date={2014},
   pages={23--31},
}

\bib{CFMN1}{article}{
   author={del Campo, R.}
   author={Fern\'andez, A.}
   author={Mayoral, F.}
   author={Naranjo, F.}
   title={Reflexivity of function sapces associated to a $\sigma$-finite vector measure},
   journal={J. Math. Anal. Appl.},
   volume={438},
   date={2016},
   pages={339--350},
}

\bib{CFMN2}{article}{
   author={del Campo, R.}
   author={Fern\'andez, A.}
   author={Mayoral, F.}
   author={Naranjo, F.}
   title={The de la Vall\'ee-Poussin theorem and Orlicz spaces associated to a vector measure},
   journal={J. Math. Anal. Appl.},
   volume={470},
   date={2019},
   pages={270--291},
}

\bib{preprint}{article}{
   author={del Campo, R.}
   author={Fern\'andez, A.}
   author={Mayoral, F.}
   author={Naranjo, F.}
   title={Compactness on $L^1$ of the semivariation of a vector measure},
   journal={preprint},
}

\bib{CR}{article}{
   author={Curbera, G.P.},
   author={Ricker, W.J.},
   title={Banach lattices with the Fatou property and optimal domains of
   kernel operators},
   journal={Indag. Math. (N.S.)},
   volume={17},
   date={2006},
   number={2},
   pages={187--204},
   issn={0019-3577},
}
\bib{DU}{book}{
    author={Diestel, J.}
    author={Uhl, J.J. Jr.}
    title={Vector measures},
    publisher={American Mathematical Society},
    series={Mathematical Surveys},
    volume={15},
    address={Providence R.I.},
    date={1977},
}

\bib{Del}{article}{
    author={Delgado, O.}
    title={Banach function subspaces of $L^1$ of a vector measure and related Orlicz spaces},
    journal={Indag. Math. (N.S.)},
    volume={15},
    number={4}
    date={2004},
    pages={485--495},
}

\bib{FMN}{article}{
   author={Fern\'andez, A.}
   author={Mayoral, F.}
   author={Naranjo, F.}
   title={Real interpolation method on spaces of scalar integrable functions with respect to vector measures},
   journal={J. Math. Anal. Appl.},
   volume={376},
   date={2011},
   pages={203--211},
}

\bib{FG}{article}{
    author={Ferrando, I.},
    author={Galaz-Fontes, F.},
    title={Multiplication operators on vector measure Orlicz spaces},
    journal={Indag. Math. (N.S.)},
    volume={20},
    number={1}
    date={2009},
    pages={57--71},
}

\bib{JPU}{article}{
    author={Jain, P.},
    author={Persson, L.E.},
    author={Upreti, P.},
    title={Inequalities and properties of some generalized Orlicz classes and spaces},
    journal={Acta Math. Hungar.},
    volume={117},
    date={2007},
    number={1--2},
    pages={161--174},
}

\bib{Kal1}{article}{
    author={Kalton, N.},
    title={Analytic functions in non-locally convex spaces},
    journal={Studia Math.},
    volume={83},
    date={1986},
    pages={275--303},
}

\bib{Kal0}{article}{
    author={Kalton, N.},
    title={Convexity conditions for non-locally convex lattices},
    journal={Glasgow Math. J.},
    volume={25},
    date={1984},
    pages={141--152},
}

\bib{Kal2}{article}{
    author={Kalton, N.},
    title={Plurisubharmonic functions on quasi-Banach spaces},
    journal={Studia Math.},
    volume={84},
    date={1986},
    pages={297--324},
}

\bib{KaMi}{article}{
    author={Kalton, N.},
    author={Mitrea, M.},
    title={Stability results on interpolation scales of quasi-Banach spaces and applications},
    journal={Trans. Amer. Math. Soc.},
    volume={350},
    number={10},
    date={1998},
    pages={3903--3922},
}

\bib{KMP}{article}{
    author={Kami\'{n}ska, A.},
    author={Maligranda, L.},
    author={Persson, L.E.},
    title={Indices, convexity and concavity of Calder\'on-Lozanovskii spaces},
    journal={Math. Scand.},
    volume={92},
    date={2003},
    pages={141--160},
}

\bib{KaAk}{book}{
    author={Kantorovich, L.V.},
    author={Akilov, G.P.},
    title={Functional Analysis},
    publisher={Pergamon press},
    date={1982},
}

\bib{KPS}{book}{
   author={Kre\u{\i}n, S.G.},
   author={Petunin, Yu. \={I}.},
   author={Semenov, E.M.},
   title={Interpolation of linear operators},
   series={Translations of Mathematical Monographs},
   volume={54},
   publisher={American Mathematical Society, Providence, R.I.},
   date={1982},
}

\bib{Mal85}{article}{
   author={Maligranda, L.},
   title={Indices and interpolation}
   journal={Dissertationes Math. (Rozprawy Mat.)},
   volume={234},
   date={1985},
   pages={1--49},
}

\bib{Mal}{article}{
   author={Maligranda, L.},
   title={Type, cotype and convexity properties of quasi-Banach spaces}
   journal={Proceedings of the International Symposium on Banach and Function Spaces},
   date={2003},
   pages={83--120}
}

\bib{ORS}{book}{
    author={Okada, S.},
    author={Ricker, W.J.},
    author={S\'anchez--P\'erez, E.A.},
    title={Optimal Domain and Integral Extension of Operators Acting in Function Spaces},
    publisher={Birkh\"{a}user-Verlag},
    series={Oper. Theory Adv. Appl.},
    volume={180},
    address={Basel},
    date={2008},
}

\end{biblist}
\end{bibdiv}
\end{document}